\documentclass[11pt,spanish,english]{amsart}
\usepackage[T1]{fontenc}
\usepackage[latin9]{inputenc}
\usepackage{geometry}
\usepackage{color}
\usepackage{enumerate}

\usepackage{babel}
\addto\shorthandsspanish{\spanishdeactivate{~<>}}

\usepackage[normalem]{ulem}
\usepackage{placeins}
\usepackage{cancel}
\usepackage{units}
\usepackage{amsthm}
\usepackage{amstext}
\usepackage{amssymb}
\usepackage{setspace}
\usepackage{stackrel}
\usepackage[unicode=true,
 bookmarks=true,bookmarksnumbered=true,bookmarksopen=true,bookmarksopenlevel=3,
 breaklinks=false,pdfborder={0 0 0}]{hyperref}

\makeatletter

\numberwithin{equation}{section}
\numberwithin{figure}{section}
\numberwithin{table}{section}
\theoremstyle{plain}

\makeatother

\newcommand{\Z}{{\mathbb Z}}
\newcommand{\Q}{{\mathbb Q}}

\newcommand{\B}{{\mathbb B}}

\newcommand{\GEN}[1]{\langle #1 \rangle}
\newcommand{\PSL}{\text{PSL}}
\newcommand{\SL}{\text{SL}}
\newcommand{\Tr}{\text{Tr}}
\newcommand{\matriz}[1]{\begin{array} #1 \end{array}}

\newcommand{\var}{{\varepsilon}}

\newcommand{\diag}{{\rm diag}}

\newcommand{\V}{\mathrm{V}}

\newcommand{\MM}[2]{\left( #1 : #2 \right)}
\newcommand{\AMM}[2]{\left| #1 : #2 \right|}

\newcommand{\TwoDot}[1]{\dot{x}}
\newcommand{\PP}[1]{{#1}'}

\newtheorem{theorem}{Theorem}[section]
\newtheorem{lemma}[theorem]{Lemma}
\newtheorem{proposition}[theorem]{Proposition}

\DeclareMathOperator{\Gal}{Gal}

\title{Zassenhaus Conjecture on torsion units holds for $\SL(2,p)$ and $\SL(2,p^2)$}
\thanks{Partially supported by Ministerio de Economía y Competitividad projects MTM2012-35240 and MTM2016-77445-P and Fondos FEDER and Fundación Séneca of Murcia 19880/GERM/15.}

\author{Ángel del Río}
\author{Mariano Serrano}
\email{adelrio@um.es}
\email{mariano.serrano@um.es}
\address{Department of Mathematics, University of Murcia.}

\subjclass[2010] {16U60, 16S34}
\date{\today}

\begin{document}

\begin{abstract}
H.J. Zassenhaus conjectured that any unit of finite order and augmentation $1$ in the integral group ring $\Z G$ of a finite group $G$ is conjugate in the rational group
algebra $\Q G$ to an element of $G$.
We prove the Zassenhaus Conjecture for the groups  $\SL(2,p)$ and $\SL(2,p^2)$ with $p$ a prime number.
This is the first infinite family of non-solvable groups for which the Zassenhaus Conjecture has been proved.
We also prove that if $G=\SL(2,p^f)$, with $f$ arbitrary and $u$ is a torsion unit of $\Z G$ with augmentation $1$ and order coprime with $p$ then $u$ is conjugate in $\Q G$ to an element of $G$.
By known results, this reduces the proof of the Zassenhaus Conjecture for this groups to prove that every unit of $\Z G$ of order multiple of $p$ and augmentation $1$ has actually order $p$.
\end{abstract}

\maketitle

\section{Introduction}\label{SectionIntroduction}

For a finite group $G$, let $\V(\Z G)$ denote the group of units of augmentation $1$ in $\Z G$.
We say that two elements of $\Z G$ are \emph{rationally conjugate} if they are conjugate in the units of $\Q G$.
The following conjecture stated by H.J. Zassenhaus \cite{Zassenhaus} (see also \cite[Section~37]{SehgalBookUnits}) has centered the research on torsion units of integral group rings during the last decades:

\begin{quote}
\textbf{Zassenhaus Conjecture:} If $G$ is a finite group then every torsion element of $\V(\Z G)$ is rationally conjugate to an element of $G$.
\end{quote}

The relevance of the Zassenhaus Conjecture is that it describes the torsion units of the integral group ring of $\Z G$ provided it holds for $G$.
Recently, Eisele and Margolis announced a metabelian counterexample to the Zassenhaus Conjecture \cite{EisMar17}.
Nevertheless, the Zassenhaus Conjecture holds for large classes of solvable groups, e.g. for nilpotent groups \cite{Weiss91}, groups possessing a normal Sylow subgroup with abelian
complement \cite{HertweckColloq} or cyclic-by-abelian groups \cite{CyclicByAbelian}.
In contrast with these results, the list of non-solvable groups for which the Zassenhaus Conjecture has been proved is very limited \cite{LutharPassi, DokuchaevJuriaansPolcino, HertweckBrauer, HertweckA6, BovdiHertweck, KonovalovM22, BachleMargolis17,delRioSerrano17}.
For example, the Zassenhaus Conjecture has only been proved for sixty-two simple groups, all of them of the form $\PSL(2,q)$ (see the proof of Theorem~C in \cite{BachleMargolisPrimaryII} and \cite{MargolisdelRioSerrano17}).

The goal of this paper is proving the following theorem:

\begin{theorem}\label{main}
Let $G=\SL(2,q)$ with $q$ an odd prime power and let $u$ be a torsion element of $\V(\Z G)$ of order coprime with $q$.
Then $u$ is rationally conjugate to an element of $G$.
\end{theorem}

As a consequence of Theorem~\ref{main} and known results we will obtain the following  theorem which provides the first positive result on the Zassenhaus Conjecture for an infinite series of non-solvable groups.

\begin{theorem}\label{SL2Result}
The Zassenhaus Conjecture holds for $\SL(2,p^f)$ with $p$ a prime number and $f\le 2$.
\end{theorem}

In Section~\ref{SectionNumberTheory} we prove a number theoretical result relevant for our arguments.
Known results on $\V(\Z G)$ and properties of $\V(\Z\; \SL(2,q))$ are collected in Section~\ref{SectionGroupTheoreticalProperties}.
A particular case of Theorem~\ref{main} is proved in Section~\ref{SectionPrimePowerOrder}.
Finally in Section~\ref{SectionProofTheorem} we prove Theorem~\ref{main}.

\section{Number theoretical preliminaries}\label{SectionNumberTheory}

We use the standard notation for the Euler totient function $\varphi$ and the Möbius function $\mu$.
Moreover, $\Z_{\ge 0}$ denotes the set of non-negative integers.
Let $n$ be a positive integer. Then $\Z_n = \Z/n\Z$, $\zeta_n$ denotes a complex primitive $n$-th root of unity, $\Phi_n(X)$ denotes the $n$-th cyclotomic polynomial, i.e. the minimal polynomial of $\zeta_n$ over $\Q$, and for a prime integer $p$ let $v_p(n)$ denote the valuation of $n$ at $p$, i.e. the maximum non-negative integer $m$ with $p^m\mid n$.
If $F/K$ is a finite field extension then $\Tr_{F/K}:F\rightarrow K$ denotes the standard trace map. We will frequently use the following formula for $d$ a divisor of $n$ \cite[Lemma~2.1]{Margolis2016}:
\begin{equation}\label{TraceRootOfUnity}
\Tr_{\Q (\zeta_{n})/\Q }(\zeta_{d})=\mu(d)\frac{\varphi(n)}{\varphi(d)}.
\end{equation}

We reserve the letter $p$ to denote a positive prime integer and for every positive integer $n$  we set
$$\PP{n} = \prod_{p\mid n} p \quad \text{ and } \quad n_p=p^{v_p(n)}.$$
If moreover $x\in \Z$ then we set
\begin{eqnarray*}
	\MM{x}{n}  &=& \text{ representative of the class of } x \text{ modulo } n \text{ in the interval } \left(-\frac{n}{2}, \frac{n}{2}\right]; \\
	\AMM{x}{n} &=&\text{ the absolute value of } \MM{x}{n};\\
	\gamma_n(x) &=& \prod_{\substack{p\mid n \\ \AMM{x}{n_p} <\frac{n_p}{2p}}} p \quad \text{and} \quad
	\bar{\gamma}_n(x) = \prod_{\substack{p\mid n \\ \AMM{x}{n_p} \leq \frac{n_p}{2p}}} p
	= \begin{cases} 2\gamma_n(x), & \text{if } \AMM{x}{n_2}=\frac{n_2}{4}; \\
		\gamma_n(x),  & \text{otherwise}.
	\end{cases}
\end{eqnarray*}

Next lemma collects two facts which follow easily from the definitions.

\begin{lemma}\label{Elementary}
	Let $p$ be a prime dividing $n$ and let $x,y \in \Z$.
	Then the following conditions hold:
	\begin{enumerate}
		\item\label{pdividinglevel} If $p\mid \bar{\gamma}_n(x)$ then $\MM{x}{\frac{n_p}{p}} \equiv x \bmod n_p$.
		\item\label{LevelStep}
		Let $d \mid \PP{n}$ such that $x \equiv y \bmod \frac{n}{d}$. If $d$ divides both $\bar{\gamma}_n(x)$ and $\bar{\gamma}_n(y)$ then $x \equiv y \bmod n$.
	\end{enumerate}
\end{lemma}

For integers $x$ and $y$ we define the following equivalence relation on $\Z$:
$$x \sim_n y \quad \Leftrightarrow \quad x \equiv \pm y \mod n.$$
We denote by $\Gamma_n$ the set of these equivalence classes.

If $x,y$ and $n$ are integers with $n>0$ then let
$$\delta_{x,y}^{(n)} = \begin{cases}
1, & \text{if } x\sim_n y  ; \\
0, & \text{otherwise};
\end{cases}  \quad\text{and}\quad
\kappa^{(n)}_x = \begin{cases}2,&\text{if }x\equiv 0 \bmod n \text{ or } x\equiv \frac{n}{2}\bmod n; \\
1,&\text{otherwise.}\end{cases}
$$
For an integer $x$ (or $x\in\Gamma_n$) we set
$$\alpha^{(n)}_x =  \zeta_n^x + \zeta_n^{-x}.$$
Observe that $\Q(\alpha^{(n)}_1)$ is the maximal real subfield of $\Q(\zeta_n)$ and $\Z[\alpha^{(n)}_1]$ is the ring of integers of $\Q(\alpha^{(n)}_1)$.
If $n\ne n_2$ then let $p_0$ denote the smallest odd prime dividing $n$. Then let
$$\B_n=\left\{x\in \Z_n : \text{ for every } p\mid n, \text{ either }
\matriz{{ll}
	\AMM{x}{n_p}> \frac{n_p}{2p} \text{ or } \\ p=2, n\ne n_2, \AMM{x}{n_2}=\frac{n_2}{4}, n_{p_0}\nmid x \text{ and }	 \MM{x}{n_2} \cdot \MM{x}{n_{p_0}} >0}
\right\}$$
and
$$\mathcal{B}_n=\begin{cases}
\{\alpha^{(n)}_b : b\in \B_n\}, & \text{if } n\ne n_2; \\
\{1\}\cup \{\alpha^{(n)}_b : b\in \B_n\}, & \text{otherwise}.
\end{cases}$$
For $b\in\B_n$ and $x\in \Z$ let
$$
\beta^{(n)}_{b,x} = \begin{cases}-1, & \text{if } n\ne n_2, \AMM{x}{n_2} = \frac{n_2}{4} \text{ and } \MM{x}{n_2} \cdot \MM{b}{n_{p_0}} <0; \\
1, & \text{otherwise.} \end{cases}
$$	

The following proposition extends Proposition~3.5 of \cite{MargolisdelRioSerrano17}. The first statement implies that $\mathcal{B}_n$ is a $\Q$-basis of $\Q(\alpha_1^{(n)})$.
For $x\in \Q(\alpha_1^{(n)})$ and $b\in \mathbb{B}$ we use $C_b(x)$ to denote the coefficient of $\alpha_b^{(n)}$ in the expression of $x$ in the basis $\mathcal{B}_n$.

\begin{proposition}\label{BaseLemma}
	Let $n$ be a positive integer.
	Then
	\begin{enumerate}
		\item $\mathcal{B}_n$ is a $\Z$-basis of $\Z[\alpha^{(n)}_1]$.
		\item\label{CoefLambda} If $b\in \B_n$ and $i\in \Z$ then
		$C_b(\alpha^{(n)}_i) = \kappa^{(n)}_i \cdot \mu(\gamma(i))\cdot \beta^{(n)}_{b,i}\cdot \delta_{b,i}^{(n/\bar{\gamma}(i))}.$
	\end{enumerate}
\end{proposition}

\begin{proof}
	We only prove the proposition in the case $n\ne n_2$, as the proof in the case $n=n_2$ is similar (actually simpler).
	It is easy to see that $|\mathcal{B}_n|\leq \frac{\varphi(n)}{2}=[\Q(\alpha^{(n)}_1):\Q]$. Thus it is enough to prove the following equality
	$\alpha^{(n)}_i=\kappa^{(n)}_i \; \mu(\gamma(i)) \sum_{b\in \B_n, b\sim_{n/\bar{\gamma}(i)} i}  \beta^{(n)}_{b,i}\;\alpha^{(n)}_b.$
	Actually we will show
	$$\zeta_n^i = \mu(\gamma(i)) \sum_{\substack{b\equiv i \bmod n/\bar{\gamma}(i) \\ b\in\B_n} } \beta^{(n)}_{b,i}\;\zeta_n^b$$
	which easily implies the desired expression of $\alpha^{(n)}_i$.

	Indeed, for every $p\mid n$ let $\zeta_{n_p}$ denote the $p$-th part of $\zeta_n$, i.e. $\zeta_{n_p}$ is a primitive $n_p$-th root of unity and $\zeta_n = \prod_{p \mid n} \zeta_{n_p}$.
	Let $J$ be the set of tuples $(j_p)_{p\mid \bar{\gamma}(i)}$ satisfying the following conditions:
	\begin{itemize}
		\item If $p\mid \gamma(i)$ then $j_p\in \{1,\dots, p-1\}$.
		\item If $p=2$ and $\AMM{i}{n_2}=\frac{n_2}{4}$ then $j_2=\begin{cases}
		1, & \text{if } \MM{i}{n_2}\cdot \MM{i+j_{p_0}\frac{n_{p_0}}{p_0}}{n_{p_0}}<0; \\ 0, &\text{otherwise}.
		\end{cases}$
	\end{itemize}
	For every $j\in J$ let $b_j\in \Z_n$ given by
	$$b_j\equiv \begin{cases}
	i+j_p\frac{n_p}{p} \mod n_p, & \text{if } p\mid \bar{\gamma}(i); \\
	i \mod n_p, & \text{otherwise}.
	\end{cases}$$
	Then $\{b_j : j\in J \}$ is the set of elements in $\B_n$ satisfying $i\equiv b \bmod \frac{n}{\bar{\gamma}(i)}$.
	From
	$$0 = \zeta_{n_p}^i \left(1 + \zeta_{n_p}^{\frac{n_p}{p}} + \zeta_{n_p}^{\frac{2n_p}{p}} + ... + \zeta_{n_p}^{\frac{(p-1)n_p}{p}} \right)$$
	we obtain $\zeta_{n_p}^i = - \sum_{j_p=1}^{p-1} \zeta_{n_p}^{i+j_p\frac{n_p}{p}}$.
	Therefore, if $\AMM{i}{n_2}\ne \frac{n_2}{4}$ then $\gamma(i)=\bar{\gamma}(i)$, $\beta^{(n)}_{b,i}=1$ for every $b\in \B_n$ and
	$$\zeta_n^i =
	\prod_{\substack{p \mid n \\ p \nmid \gamma(i)}} \zeta_{n_p}^i \prod_{\substack{p \mid n \\ p \mid \gamma(i)}}
	\left(-\sum_{j_p=1}^{p-1 }  \zeta_{n_p}^{i+j_p\frac{n_p}{p}} \right) =
	\mu(\gamma(i)) \sum_{j\in J} \zeta_n^{b_j} =
	\mu(\gamma(i)) \sum_{\substack{b\equiv i \bmod n/\bar{\gamma}(i) \\ b\in\B_n} } \zeta_n^b.$$
	This gives the desired equality in this case.
	
	Suppose that $\AMM{i}{n_2}=\frac{n_2}{4}$. Then $\zeta_{n_2}^i = \beta^{(n)}_{b_j,i} \zeta_{n_2}^{b_j}$ for every $j\in J$.
	Then a small modification of the argument in the previous paragraph gives
	$$\zeta_n^i = \zeta_{n_2}^i \prod_{\substack{p \mid n \\ p \nmid \gamma(i)}} \zeta_{n_p}^i \prod_{\substack{p \mid n \\ p \mid \gamma(i)}}
	\left(-\sum_{j_p=1}^{p-1 }  \zeta_{n_p}^{i+j_p\frac{n_p}{p}} \right) =
	\mu(\gamma(i)) \sum_{j\in J} \beta^{(n)}_{b,i}\;\zeta_n^{b_j} =
	\mu(\gamma(i)) \sum_{\substack{b\equiv i \bmod n/\bar{\gamma}(i) \\ b\in\B_n} } \beta^{(n)}_{b,i}\;\zeta_n^b.$$
\end{proof}

\section{Group theoretical preliminaries}\label{SectionGroupTheoreticalProperties}

Let $G$ be a finite group.
We denote by $Z(G)$ the center of $G$.
If $g\in G$ then $|g|$ denotes the order of $g$, $\langle g\rangle$ denotes the cyclic group generated by $g$,
and $g^{G}$ denotes the conjugacy class of $g$ in $G$.
If $R$ is a ring then $RG$ denotes the group ring of $G$ with coefficients in $R$.
If $\alpha=\sum_{g\in G} \alpha_g g$ is an element of a group ring $RG$, with each $\alpha_g\in R$, then the partial augmentation of $\alpha$ at $g$ is defined as
	$$\var_{g}(\alpha)=\sum_{h\in g^{G}}\alpha_{h}.$$

We collect here some known results on partial augmentations of an element $u$ of order $n$ in $\V(\Z G)$.
\begin{enumerate}[(A)]
\item\label{BermanHigman}
\cite[Proposition~1.5.1]{EricAngel1}
(Berman-Higman Theorem).
If $g\in Z(G)$ and $u\ne g$ then $\varepsilon_g(u)=0$.

\item\label{OrdersDivides} \cite[Theorem~2.3]{HertweckBrauer} If $g\in G$ and $\varepsilon_g(u)\ne 0$ then $|g|$ divides $n$.

\item\label{MRSW}
\cite[Theorem~2.5]{MarciniakRitterSehgalWeiss}
$u$ is rationally conjugate to an element of $G$ if and only if $\varepsilon_g(u^d) \geq 0$ for all $g \in G$ and all divisors $d$ of $n$.

\item\label{equationHELP}
\cite{LutharPassi,HertweckBrauer}
Let $F$ be a field of characteristic $t\ge 0$ with $t\nmid n$.
Let $\rho$ be an $F$-representation of $G$.
If $t\ne 0$ then let $\xi_n$ be a primitive $n$-th root of unity in $F$, so that if $t=0$ then $\xi_n=\zeta_n$.
Let $T$ be a set of representatives of the conjugacy classes of $t$-regular elements of $G$ (all the conjugacy classes if $t=0$).
Let $\chi$ denote the character afforded by $\rho$ if $t=0$, and the $t$-Brauer character of $G$ afforded by $\rho$ if $t>0$ (using a group isomorphism associating $\xi_n$ to $\zeta_n$).
Then for every integer $\ell$, the multiplicity of $\xi_n^{\ell}$ as eigenvalue of $\rho(u)$ is
$$\frac{1}{n}\sum_{x\in T}\sum_{d\mid n}\varepsilon_x(u^d)\Tr_{\Q(\zeta_n^d)/\Q}(\chi(x)\zeta_n^{-\ell d})$$
\end{enumerate}

In the remainder of the paper we fix an odd prime power $q$ and let $G=\SL(2,q)$, $\overline{G}=\PSL(2,q)$ and let  $\pi:G\rightarrow \overline{G}$ denote the natural projection, which we extend by linearity to a ring homomorphism
$\pi:\Z G \rightarrow \Z \overline{G}$.

We collect some group theoretical properties of $G$ and $\overline{G}$ (see e.g. \cite[Theorem~38.1]{Dornhoff1971}).

\begin{enumerate}[(A)]
\setcounter{enumi}{4}
\item\label{CCSL}
$G$ has a unique element $J$ of order $2$ and $q+4$ conjugacy classes.
More precisely, if $p$ is the prime dividing $q$ then $2$ of the classes are formed by elements of order $p$, another $2$ are formed by elements of order $2p$ and $q$ classes are formed by elements of order dividing $q+1$ or $q-1$.
Furthermore, if $g$ and $h$ are $p$-regular elements of $G$ and $|h|$ divides $|g|$ then $h$ is conjugate in $G$ to an element of $\GEN{g}$ and two elements of $\GEN{g}$ are conjugate in $G$ if and only if they are equal or mutually inverse.

\item\label{CCPSL}
Let $C$ be a conjugacy class of $\overline{G}$ formed by elements of order $n$.
If $n=2$ then $\pi^{-1}(C)$ is the only conjugacy class of $G$ formed by elements of order $4$.
Otherwise,  $\pi^{-1}(C)$ is the union of two conjugacy classes $C_1$ and $C_2$ of $G$ with $C_2=JC_1$.
Furthermore, if $n$ is multiple of $4$ then the elements of $C_1$ and $C_2$ have order $2n$ while if $n$ is not multiple of $4$ then one of the classes $C_1$ or $C_2$ is formed by elements of order $n$.
\end{enumerate}

The following proposition collects some consequences of these facts.

\begin{proposition}\label{Easy}
Let $u$ be a torsion element of $\V(\Z G)$ and let $n=|u|$.
\begin{enumerate}
\item\label{Order2} $J$ is the unique element of order $2$ in $\V(\Z G)$.

\item\label{OrderPiu} $|\pi(u)|=\frac{n}{\gcd(2,n)}$.

\item\label{LiftingNot4} If $4\nmid n$ and $\pi(u)$ is rationally conjugate to an element of $\overline{G}$ then $u$ is rationally conjugate to an element of $G$.

\item\label{NotMultiple4}
If $\gcd(n,q)=1$ and $4\nmid n$ then $u$ is rationally conjugate to an element of $G$.

\item\label{Spectrum} If $\gcd(n,q)=1$ then $G$ has an element of order $n$.

\item\label{Multiple-p}
Suppose that $q=p^f$ with $p$ prime, $f\le 2$ and $p\mid n$. Then $u$ is rationally conjugate to an element of $G$.

\item\label{MultiplicityInverses} If $\rho$ is a representation of $G$ and $\zeta$ is a root of unity of order dividing $n$ then $\zeta$ and $\zeta^{-1}$ have the same multiplicity as eigenvalues of $\rho(u)$.
\end{enumerate}
\end{proposition}

\begin{proof}
\eqref{Order2}  is a direct consequence of \eqref{BermanHigman} and \eqref{CCSL}.

\eqref{OrderPiu} By the main result of \cite{MargolisHertweck}, if $\pi(u)=1$ then $u^2=1$ and hence either $u=1$ or $u=J$, by \eqref{Order2}.
Then \eqref{OrderPiu} follows.
	
\eqref{LiftingNot4}
Suppose that $n$ is not multiple of $4$.
If $n$ is even  then the order of $Ju$ is odd, by \eqref{Order2}.
Thus, we may assume without loss of generality that the order of $u$ is odd.
If $\varepsilon_g(u)\ne 0$ then $|g|$ is odd, by \eqref{OrdersDivides}, and hence $\varepsilon_g(u)=\varepsilon_{\pi(g)}(\pi(u))\ge 0$, by \eqref{CCPSL}.
Thus $u$ is rationally conjugate to an element of $G$.

\eqref{NotMultiple4}
Let $q=p^f$ where $p$ is an odd prime number.
If $p\nmid n$ and $4\nmid n$ then $|\pi(u)|$ is coprime with $2q$, by (\ref{OrderPiu}), and hence $\pi(u)$ is rationally conjugate to an element of $\overline{G}$, by  \cite[Theorem~1.1]{MargolisdelRioSerrano17}.
Then $u$ is rationally conjugate to an element of $G$ by \eqref{LiftingNot4}.

\eqref{Spectrum} is a consequence of \eqref{OrderPiu} and \cite[Proposition~6.7]{HertweckBrauer}.

\eqref{Multiple-p}
In this case $|\pi(u)|=p$ by \eqref{OrderPiu} and \cite[Theorem~A]{BachleMargolisPrimaryI}.
Then $n$ is either $p$ or $2p$, by \eqref{OrderPiu}, and $\pi(u)$ is rationally conjugate to an element of $\overline{G}$, by \cite[Proposition~6.1]{HertweckBrauer}.
Thus $u$ is rationally conjugate to an element of $G$, by \eqref{LiftingNot4}.

\eqref{MultiplicityInverses} is a consequence of \eqref{CCSL} and the formula in \eqref{equationHELP}.
\end{proof}

Observe that for $q$ odd, Theorem~\ref{SL2Result} follows at once from Theorem~\ref{main} and statement \eqref{Multiple-p} of Proposition~\ref{Easy}.
On the other hand $\SL(2,2)\cong S_3$ and $\SL(2,4)\cong A_5$ for which the Zassenhaus Conjecture is well known.
So in the remainder of the paper we concentrate on proving Theorem~\ref{main}.
For that from now on $t$ denotes the prime dividing $q$ (we want to set free the letter $p$ to denote an arbitrary prime) and we introduce some $t$-Brauer characters of $G$.

Let $g$ be an element of $G$ of order $n$ with $t\nmid n$ and let $\xi_n$ denote a primitive $n$-th root of unity in a field $F$ of characteristic $t$.
Adapting the proof of \cite[Lemma~1.2]{Margolis2016} we deduce that for every positive integer $m$ there is an $F$-representation $\Theta_m$ of $G$ of degree $1+m$ such that
    \begin{equation}\label{SL2ModularCharacters}
    \Theta_{m}(g)=
    \begin{cases}
    \diag\left(1, \xi_n^2, \xi_n^{-2}, \ldots, \xi_n^{m}, \xi_n^{-m} \right), & \text{if }2\mid m;\\
    \diag\left(\xi_n, \xi_n^{-1}, \xi_n^{3}, \xi_n^{-3}, \ldots, \xi_n^{m}, \xi_n^{-m} \right), & \text{if }2\nmid m.\\
    \end{cases}
    \end{equation}
In particular, the restriction to $\GEN{g}$ of the $t$-Brauer character associated to $\Theta_m$ is given by
    \begin{equation*}
    \chi_m(g^i)=\sum_{\substack{j=-m\\j\equiv m \bmod 2}}^{m}\zeta_{n}^{ij}.
    \end{equation*}

\section{Prime power order}\label{SectionPrimePowerOrder}

In this section we prove the following particular case of Theorem~\ref{main}.

\begin{proposition}\label{PropositionPrimePowerOrder}
Let $G=\SL(2,q)$ with $q$ an odd prime power and let $u$ be a torsion element of $\V(\Z G)$.
If the order of $u$ is a prime power and it is coprime with $q$ then $u$ is rationally conjugate to an element of $G$.
\end{proposition}

\begin{proof}
By Proposition~\ref{Easy}.\eqref{NotMultiple4} we may assume that $|u|=2^r$ with $r\ge 3$.
We argue by induction on $r$.
So we assume that units of order $2^k$ with $1\leq k \leq r-1$ are rationally conjugate to an element of $G$.
By Proposition~\ref{Easy}.\eqref{Spectrum} and \eqref{CCSL}, $G$ has an element $g_0$ of order $2^r$ such that
$\{g_0^k : k=0,1,2,\ldots , 2^{r-1}\}$ is a set of representatives of the conjugacy classes of $G$ with order a divisor of $2^r$.
By \eqref{OrdersDivides}, the only possible non-zero partial augmentations of $u$ are the integers $\varepsilon_k=\varepsilon_{g_0^k}(u)$, with $k=1,\dots,2^{r-1}-1$.
By the induction hypothesis, if $1\le i\le r$ and $g\in G$ then $\varepsilon_g(u^{2^i})\ge 0$. 
Hence,  by \eqref{MRSW}, it suffices to prove that $\varepsilon_k=0$ for all but one $k=0,1,\dots,2^{r-1}$.

By \cite[Theorem~2]{Margolis2016} and Proposition~\ref{Easy}.\eqref{OrderPiu}, $\pi(u)$ is rationally conjugate to an element of order $2^{r-1}$ in $\overline{G}$ and hence $\varepsilon_{2^{r-2}}=\varepsilon_{\pi(g_0)^{2^{r-2}}}(\pi(u))=0$, by \eqref{CCPSL}.

For a $t$-Brauer character $\chi$ of $G$ and an integer $\ell$ define
$$A(\chi,\ell)=\sum_{k=1}^{2^{r-1}-1}\varepsilon_k\cdot \Tr_{\Q(\zeta_{2^r})/\Q}\left(\chi(g_0^k)\cdot \zeta_{2^r}^{-\ell}\right) \quad \text{and}\quad
B(\chi,\ell)=\sum_{k=0}^{r-1}\Tr_{\Q(\zeta_{2^k})/\Q}\left(\chi(g_0^{2^{r-k}}) \cdot \zeta_{2^k}^{-\ell}\right).$$
Then, by \eqref{equationHELP}, we have
\begin{equation}\label{LPPower2}
\frac{1}{2^r}\left(A(\chi,\ell) + B(\chi,\ell)\right)\in \Z_{\geq 0}.
\end{equation}
Observe that $B(\chi,\ell+2^{r-1})=B(\chi,\ell)$ and $A(\chi,\ell+2^{r-1})=-A(\chi,\ell)$.
Therefore, from \eqref{LPPower2} it follows that
\begin{align}
\label{B=0} & \text{if } B(\chi,\ell)=0  \text{ then }  A(\chi,\ell)=0; \\
\label{B=2^{r-1}} & \text{if } B(\chi,\ell)=2^{r-1} \text{ then }  A(\chi,\ell)=\pm 2^{r-1}.
\end{align}

We will calculate $B(\chi,\ell)$ and $A(\chi,\ell)$ for several $t$-Brauer characters $\chi$ and several integers $\ell$ and for that we will use \eqref{TraceRootOfUnity} without further mention. We start proving that
	\begin{equation}\label{B(2^h0)}
	\text{ if } 0\leq h \leq r-2 \text{ and } 2^{r-1}\mid \ell \text{ then }B(\chi_{2^h}, \ell) = \begin{cases}2^{r-1}, & \text{if }h\geq 1;\\ 0, & \text{if }h=0;\end{cases}
	\end{equation}
and
	\begin{equation}\label{B(2^hk)}
		\text{if } 0\leq h \leq r-3, \; 2^h \mid \ell \text{ and } 2^{r-1}\nmid \ell \text{ then } B(\chi_{2^h}, \ell) = \begin{cases}2^{r-1}, & \text{if } \ell\equiv \pm 2^h \bmod 2^{r-1}; \\ 0,& \text{otherwise.} \end{cases}
	\end{equation}
In both cases we argue by induction on $h$ with the cases $h=0$ and $h=1$ being straightforward.
Suppose that $1<h\le r-2$, $2^{r-1}\mid \ell$ and $B(\chi_{2^{h-1}},\ell)=2^{r-1}$.
If $j$ is even, then straightforward calculations show that $\sum_{k=0}^{r-1} \Tr_{\Q (\zeta_{2^k})/\Q} \left(\zeta_{2^k}^{2^{h-1}+j} + \zeta_{2^k}^{-2^{h-1}-j} \right)=0$.  This implies that
	$B(\chi_{2^h},\ell) = B(\chi_{2^{h-1}},\ell) + \sum_{\substack{j=2 \\ 2\mid j}}^{2^{h-1}}\sum_{k=0}^{r-1} \Tr_{\Q (\zeta_{2^k})/\Q} \left(\zeta_{2^k}^{2^{h-1}+j} + \zeta_{2^k}^{-2^{h-1}-j} \right)=2^{r-1}$. This finishes the proof of \eqref{B(2^h0)}.

Suppose that $1<h \leq r-3$, $2^h \mid \ell$ and $2^{r-1}\nmid \ell$. In this case the induction hypothesis implies    $B(\chi_{2^{h-1}},\ell)=0$.
Arguing as in the previous paragraph we get
    $B(\chi_{2^h},\ell) =
    \sum_{j=2,2\mid j}^{2^{h-1}}  \; \sum_{k=0}^{r-1} \Tr_{\Q (\zeta_{2^k})/\Q}
    \left(\left(\zeta_{2^k}^{2^{h-1}+j} + \zeta_{2^k}^{-(2^{h-1}+j)}\right)\zeta_{2^k}^{-\ell}\right)$.
However, if $j$ is even and smaller than $2^{h-1}$ then
$\sum_{k=0}^{r-1} \Tr_{\Q (\zeta_{2^k})/\Q}
\left(\left(\zeta_{2^k}^{2^{h-1}+j} + \zeta_{2^k}^{-(2^{h-1}+j)}\right)\zeta_{2^k}^{-\ell}\right)=0$.
Therefore, having in mind that $\zeta_{2^{h+2}}^{2^h} + \zeta_{2^{h+2}}^{-2^h}=0$ we have
$$B(\chi_{2^h},\ell) =
		\sum_{k=0}^h \Tr_{\Q(\zeta_{2^k})/\Q} \left( \left( \zeta_{2^k}^{2^h} + \zeta_{2^k}^{-2^h} \right)\zeta_{2^k}^{-\ell} \right) + \epsilon 2^{h+1} +
		\sum_{k=h+3}^{r-1} \Tr_{\Q(\zeta_{2^k})/\Q} \left( \left( \zeta_{2^k}^{2^h} + \zeta_{2^k}^{-2^h} \right)\zeta_{2^k}^{-\ell} \right),
$$
where $\epsilon=1$ if $2^{h+1}\nmid \ell$ and $\epsilon=-1$ otherwise.
Then the claim follows using the following equalities that can be proved by straightforward calculations: $\sum_{k=0}^h \Tr_{\Q(\zeta_{2^k})/\Q} \left( \left( \zeta_{2^k}^{2^h} + \zeta_{2^k}^{-2^h} \right)\zeta_{2^k}^{-\ell} \right)=2^{h+1}$ and
	$$\sum_{k=h+3}^{r-1} \Tr_{\Q(\zeta_{2^k})/\Q} \left( \left( \zeta_{2^k}^{2^h} + \zeta_{2^k}^{-2^h} \right)\zeta_{2^k}^{-\ell} \right) =
	\begin{cases}
	0, & \text{if } 2^{h+1}\mid \ell;\\
	2^{r-1} - 2^{h+2}, & \text{if }2^{h+1}\nmid \ell \text{ and } \ell\equiv \pm 2^h \bmod 2^{r-1};\\
	-2^{h+2}, & \text{if } 2^{h+1}\nmid \ell \text{ and } \ell\not\equiv \pm 2^h \bmod 2^{r-1}.
	\end{cases}
	$$
This finishes the proof of \eqref{B(2^hk)}.

We now prove, by induction on $h$, that the following two statements hold for any integer $0\leq h \leq r-3$:
\begin{align}
\label{Statement 1-1}
    & \sum_{k\in X_h} (\varepsilon_k - \varepsilon_{k+2^{r-h-1}})=\pm 1, \text{ where } X_h=\{i\in \{1,\dots,2^{r-2}\} : i\equiv \pm 1 \bmod 2^{r-h}\}; \\
\label{Statement i=j}
    & \text{if } i\equiv \pm j \bmod 2^{r-h-1} \text{ and } i\not\equiv 0,\pm 1\bmod 2^{r-h-1} \text{ then } \varepsilon_i = \varepsilon_j;
\end{align}
and that the next one holds for every $0\le h\le r-2$:
\begin{equation}\label{Statement i=0}
    \text{if } i\equiv 0 \bmod 2^{r-h-1} \text{ then } \varepsilon_i=0.
\end{equation}

Observe that $X_0=\{1\}$.
Fix an integer $i$. Then for every integer $k$ we have
		$$\Tr_{\Q(\zeta_{2^r})/\Q}\left(\left(\zeta_{2^r}^k + \zeta_{2^r}^{-k}\right)\zeta_{2^r}^{-i}\right) =
		\begin{cases}  2^{r-1}, & \text{if } k\equiv i \mod 2^r;\\
		- 2^{r-1}, & \text{if } k \equiv 2^{r-1} -i \mod 2^r;\\
		0,& \text{otherwise}.
		\end{cases}$$
Thus $A(\chi_1, i) = 2^{r-1} \left(\varepsilon_i - \varepsilon_{i+2^{r-1}}\right)$ and hence for $h=0$, \eqref{Statement 1-1} and \eqref{Statement i=j} follows at once from \eqref{B=0}, \eqref{B=2^{r-1}} and \eqref{B(2^hk)}.
Moreover, for $h=0$ \eqref{Statement i=0} is clear because $\varepsilon_{2^{r-1}}=0$.

Suppose $0<h\le r-3$ and (\ref{Statement 1-1}), (\ref{Statement i=j}) and (\ref{Statement i=0}) hold for $h$ replaced by $h-1$. Suppose also that $i\not\equiv 0\bmod 2^{r-h-1}$. To prove (\ref{Statement 1-1}) and (\ref{Statement i=j}) we first compute $A(\chi_{2^h},2^hi)$ which we split in three summands:
\begin{eqnarray*}
A(\chi_{2^h},2^hi) &=& \sum_{k=1}^{2^{r-1}-1} \varepsilon_k \Tr_{\Q(\zeta_{2^r})/\Q}(\zeta_{2^r}^{-2^hi})
    + \sum_{\substack{j=2\\2\mid j}}^{2^h-2} \sum_{k=1}^{2^{r-1}-1} \varepsilon_k \Tr_{\Q(\zeta_{2^r})/\Q}\left(\left(\zeta_{2^r}^{kj}+\zeta_{2^r}^{-kj}\right)\zeta_{2^r}^{-2^hi}\right) \\
    &&+ \sum_{k=1}^{2^{r-1}-1} \varepsilon_k
        \Tr_{\Q(\zeta_{2^r})/\Q}\left( \zeta_{2^r}^{2^h(k-i)} +\zeta_{2^r}^{-2^h(k+i)}\right).
\end{eqnarray*}
We now prove that the first two summands are $0$.
This is clear for the first one because $2^{r-1}\nmid 2^hi$.
To prove that the second summand is $0$ let $2\le j \le 2^h-2$ and $2\mid j$.
Observe that $2^h\nmid j$. Thus, if $k$ is odd then the order of $\zeta_{2^r}^{\pm kj-2^hi}$ is multiple of $2^{r-h-1}$ and, as $h\le r-3$, we deduce that
$\Tr_{\Q(\zeta_{2^r})/\Q}\left(\zeta_{2^r}^{kj-2^hi}\right)=\Tr_{\Q(\zeta_{2^r})/\Q}\left(\zeta_{2^r}^{-kj-2^hi}\right)=0$.
Thus we only have to consider the summands with $k$ even.
Actually we can exclude also the summands with $2^{r-h}\mid k$ because, by the induction hypothesis on \eqref{Statement i=0}, for such $k$ we have $\varepsilon_k=0$.
For the remaining values of $k$ (i.e. $k$ even and not multiple of $2^{r-h}$) we have $\varepsilon_k=\varepsilon_l$ if $k\equiv l \bmod 2^{r-h-1}$, by the induction hypothesis on (\ref{Statement i=j}).
So, we can rewrite $\sum_{k=1}^{2^{r-1}-1} \varepsilon_k \Tr_{\Q(\zeta_{2^r})/\Q}\left(\left(\zeta_{2^r}^{kj}+\zeta_{2^r}^{-kj}\right)\zeta_{2^r}^{-2^hi}\right)$ as
	$$\sum_{l\in \Z_{2^{r-h-1}}} \varepsilon_l \Tr_{\Q(\zeta_{2^r})/\Q}\left( \zeta_{2^r}^{l-2^hi}\left(\sum_{a=0}^{2^h-1} (\zeta_{2^r}^{2^{r-h-1}j})^a\right)+
    \zeta_{2^r}^{-l-2^hi}\left(\sum_{a=0}^{2^h-1} (\zeta_{2^r}^{-2^{r-h-1}j})^a\right)\right),$$
which is $0$ because $\zeta_{2^r}^{2^{r-h-1}j}$ is a root of unity different from $1$ and of order dividing $2^h$, as $j$ is even but not multiple of $2^h$.
This finishes the proof that the first two summands are $0$.
To finish the calculation of $A(\chi_{2^h},2^hi)$ we compute
$$
\Tr_{\Q(\zeta_{2^r})/\Q}\left( \zeta_{2^r}^{2^h(k-i)} +\zeta_{2^r}^{-2^h(k+i)}\right) =
\begin{cases}
2^{r-1}, & \text{if } k\in X_{h,i}; \\
- 2^{r-1}, & \text{if } k-2^{r-1} \in X_{h,i}; \\
0, & \text{otherwise,}
\end{cases}
$$
where $X_{h,i} = \{k\in\{1,\dots,2^{r-2}\} : k\equiv \pm i \bmod 2^{r-h}\}$.
So we have proved the following:
$$A(\chi_{2^h}, 2^hi) = 2^{r-1}\sum_{k\in X_{h,i}}(\varepsilon_k - \varepsilon_{k+2^{r-h-1}}).$$
Then \eqref{Statement 1-1} follows from \eqref{B=2^{r-1}}, \eqref{B(2^hk)} and the previous formula.
Using \eqref{B=0} we also obtain that $\sum_{k\in X_{h,i}}\varepsilon_k = \sum_{k\in X_{h,i}}\varepsilon_{k+2^{r-h-1}}$ if $i\not\equiv \pm 1 \bmod 2^{r-h-1}$.
However, in this case the induction hypothesis for \eqref{Statement i=j} means that the $\varepsilon_k$ with $k\in X_{h,i}$ are all equal and the $\varepsilon_{k+2^{r-h-1}}$ with $k\in X_{h,i}$ are all equal. Hence \eqref{Statement i=j} follows.

In order to deal with \eqref{Statement i=0}, assume that $0<h\le r-2$. By induction hypothesis
on \eqref{Statement i=0} we have $\varepsilon_k=0$ if $2^{r-h}\mid k$, and by the induction
hypothesis on \eqref{Statement i=j}, we have that $\varepsilon_k$ is constant on the set $X$ formed by integers $1\le k \le 2^{r-1}$ such that $k\equiv 2^{r-h-1} \mod 2^{r-h}$.
We will use these two facts without specific mention.
Arguing as before we have
\begin{eqnarray*}
	 A(\chi_{2^h},0) &=& \sum_{k=1}^{2^{r-1}-1}\varepsilon_k\Tr_{\Q(\zeta_{2^r})/\Q}\left( 1+\zeta_{2^r}^{2^hk} + \zeta_{2^r}^{-2^hk} \right) +
	\sum_{k=1}^{2^{r-1}-1}\varepsilon_k\Tr_{\Q(\zeta_{2^r})/\Q}\left( \sum_{j=1}^{2^{h-1}-1}\left(\zeta_{2^r}^{2jk} + \zeta_{2^r}^{-2jk}\right)\right) \\
	& =& \sum_{k=1,2^{r-h}\nmid k}^{2^{r-1}-1}\varepsilon_k\Tr_{\Q(\zeta_{2^r})/\Q}\left( 1+\zeta_{2^r}^{2^hk} + \zeta_{2^r}^{-2^hk} \right).
\end{eqnarray*}
As
$$
\Tr_{\Q(\zeta_{2^r})/\Q}\left( 1+\zeta_{2^r}^{2^hk} + \zeta_{2^r}^{-2^hk} \right) =
\begin{cases}
2^{r-1}, & \text{if } 2^{r-h-1} \nmid k;\\
- 2^{r-1}, & \text{if } 2^{r-h-1} \mid k \text{ and } 2^{r-h}\nmid k;\\
\end{cases}
$$
we obtain
	$$A(\chi_{2^h},0) = 2^{r-1}
	\left( \sum_{2^{r-h-1}\nmid k}\varepsilon_k  - \sum_{2^{r-h-1} \mid k} \varepsilon_k\right) =
	2^{r-1} \left(1-2\sum_{k\in X} \varepsilon_k\right) =
	2^{r-1} \left(1-2|X| \varepsilon_k\right).$$
From \eqref{B=2^{r-1}} and \eqref{B(2^h0)} we deduce that if $k\in X$ then $1-2|X| \varepsilon_k=\pm 1$ and hence $\varepsilon_k=0$, since  $|X|=2^{r-h-1}\ge 2$, as $h\le r-2$. This finishes the proof of \eqref{Statement i=0}.

To finish the proof of the proposition it is enough to show that $\varepsilon_i\ne 0$ for exactly one  $i\in \{1,\dots,2^{r-1}-1\}$.
If $i$ is even then $\varepsilon_i=0$, by \eqref{Statement i=0} with $h=r-2$.

We claim that if $\varepsilon_i\ne 0$ then $i\equiv \pm 1 \bmod 2^{r-1}$.
Otherwise, there are integers $2\le v\le r-2$ and $2<i<2^{r-1}-1$ satisfying $i\not \equiv \pm 1 \bmod 2^{v+1}$ and $\varepsilon_i\ne 0$.
We choose $v$ minimum with this property for some $i$.
Then (1) $\varepsilon_k=0$ for every $k\not\equiv \pm 1 \bmod 2^v$  and (2) $i\equiv \pm (k+2^v) \bmod 2^{v+1}$ for every $k\in X_{r-v-1}$.
(1) implies that  $\sum_{k\in X_{r-v-1}}(\varepsilon_k+\varepsilon_{k+2^v})=1$.
On the other hand $1\le r-v-1\le r-3$ and hence applying \eqref{Statement 1-1} and \eqref{Statement i=j} with $h=r-v-1$ we deduce from (2) that
$\varepsilon_i= \varepsilon_{k+2^v}$ for every $k\in X_{r-v-1}$ and
$\sum_{k\in X_{r-v-1}}(\varepsilon_k-\varepsilon_{k+2^v})=\pm 1$.
Using $|X_{r-v-1}|=2^{r-v-1}$ and $\varepsilon_i\ne 0$ we deduce that  $2^{r-v}\varepsilon_i=2\sum_{k\in X_{r-v-1}} \varepsilon_{k+2^{r-v}}  = 2$, in contradiction with $2\le r-v$.
This finishes the proof of the claim.

Then the only possible non-zero partial augmentations of $u$ are $\varepsilon_1$ and $\varepsilon_{2^{r-1}-1}$.
Hence $\varepsilon_1+\varepsilon_{2^{r-1}-1}=1$ and, by applying \eqref{Statement 1-1} with $h=0$ we deduce that $\varepsilon_1-\varepsilon_{2^{r-1}-1}=\pm 1$.
Therefore, either $\varepsilon_1=0$ or $\varepsilon_{2^{r-1}-1}=0$, i.e.
$\varepsilon_i\ne 0$ for exactly one $i\in \{1,\dots,2^{r-1}-1\}$, as desired.
\end{proof}

\section{Proof of Theorem~\ref{main}}\label{SectionProofTheorem}

In this section we prove Theorem~\ref{main}.
Recall that $G=\SL (2,q)$ with $q=t^f$ and $t$ an odd prime number, $\overline{G}=\PSL(2,q)$, $\pi:G\rightarrow \overline{G}$ is the natural projection and $u$ is an element of order $n$ in $\V(\Z G)$ with $\gcd(n,q)=1$.
We have to show that $u$ is rationally conjugate to an element of $G$.
By Proposition~\ref{Easy}.\eqref{NotMultiple4}, we may assume that $n$ is multiple of $4$ and by Proposition~\ref{PropositionPrimePowerOrder} that $n$ is not a  prime power.
Moreover, we may also assume that $n\ne 12$ because this case follows easily from known results and the HeLP Method.
Indeed, if $n=12$ then $\pi(u)$ has order $6$, by Proposition~\ref{Easy}.\eqref{OrderPiu} and hence  $\pi(u)$ is rationally conjugate to an element of $\overline{G}$, by \cite[Proposition~6.6]{HertweckBrauer}.
Using this and the fact that  $G$ has a unique conjugacy class for each of the orders $3$, $4$ or $6$ and two conjugacy classes of elements of order $12$, and applying \eqref{equationHELP} with $\chi=\chi_1$ and $\ell=1,5$ it easily follows that all the partial augmentations of $u$ are non-negative.

In the remainder we follow the strategy of the proof of the main result of \cite{MargolisdelRioSerrano17}.
The difference with the arguments of that paper is twofold: On the one hand, in our case $n$ is even (actually multiple of $4$) and this introduces some difficulties not appearing in \cite{MargolisdelRioSerrano17} where $n$ was odd.
On the other hand for $\SL(2,q)$ we have more Brauer characters than for $\PSL(2,q)$ and this will help to reduce some cases.

As the order $n$ of $u$ is fixed throughout, we simplify the notation of the Section~\ref{SectionNumberTheory} by setting
    $$\gamma=\gamma_n, \quad \bar{\gamma} = \bar{\gamma}_n, \quad \alpha_x=\alpha^{(n)}_x, \quad \kappa_x=\kappa_{x}^{(n)}, \quad \beta_{b,x}=\beta^{(n)}_{b,x}, \quad \B=\B_n, \quad \mathcal{B}=\mathcal{B}_n.$$

We argue by induction on $n$.
So we also assume that $u^d$ is rationally conjugate to an element of $G$ for every divisor $d$ of $n$ with $d\ne 1$.

We will use the representations $\Theta_m$ and $t$-Brauer characters $\chi_m$ introduced in \eqref{SL2ModularCharacters}.
Observe that the kernel of $\Theta_m$ is trivial if $m$ is odd, and otherwise it is the center of $G$.
Using this and the induction hypothesis on $n$ it easily follows that the order of $\Theta_m(u)$ is $\frac{n}{2}$ if $m$ is even, while, if $m$ is odd then the order of $\Theta_m(u)$ is $n$.
Combining this with Proposition~\ref{Easy}.\eqref{MultiplicityInverses} we deduce that $\Theta_{1}(u)$ is conjugate to $\diag(\zeta, \zeta^{-1})$ for a suitable primitive $n$-th root of unity $\zeta$.
Hence there exists an element $g_{0}\in G$ of order $n$ such that $\Theta_{1}(g_{0})$ and $\Theta_{1}(u)$ are conjugate.
The element $g_0 \in G$ and the primitive $n$-th root of unity $\zeta$ will be fixed throughout and
from now on we abuse the notation and consider $\zeta$ both as a primitive $n$-th root of unity in a field of characteristic $t$ and as a complex primitive $n$-th root of unity.
Then
\begin{equation*}
\Theta_{m}(g_{0})\text{ is conjugate to } \begin{cases}
\diag\left(1, \zeta^2, \zeta^{-2}, \ldots, \zeta^{m}, \zeta^{-m} \right), & \text{if }2\mid m;\\
\diag\left(\zeta, \zeta^{-1}, \zeta^{3}, \zeta^{-3}, \ldots, \zeta^{m}, \zeta^{-m} \right), & \text{if }2\nmid m;\\
\end{cases}
\end{equation*}
and
\begin{equation}\label{ChiExpresion}
\chi_m(g_{0}^{i})=\sum_{\substack{j=-m\\j\equiv m \bmod 2}}^{m}\zeta^{ij} =
\begin{cases}
1+\alpha_{2i}+\alpha_{4i}+\dots+\alpha_{mi}, & \text{if } 2\mid m; \\
\alpha_{i}+\alpha_{3i}+\dots+\alpha_{mi}, & \text{if } 2\nmid m.
\end{cases}
\end{equation}
By the induction hypothesis on $n$, if $c$ is a divisor of $n$ with $c\ne 1$ then $u^c$ is rationally conjugate to $g_0^i$ for some $i$ and hence $\zeta^c=\zeta^{\pm i}$. Therefore $c\sim_n i$ and hence $u^c$ is conjugate to $g_0^c$.

By \eqref{CCSL}, two elements of $\GEN{g_0}$ are conjugate in $G$ if and only if they are equal or mutually inverse.
Moreover, every element $g\in G$, with $g^n=1$, is conjugate to an element of $\GEN{g_0}$.
Therefore $x\mapsto (g_0^x)^G$ induces a bijection from $\Gamma_n$ to the set of conjugacy classes of $G$ formed by elements of order dividing $n$.
For an integer $x$ (or $x\in\Gamma_n$) we set
	$$\varepsilon_x = \varepsilon_{g_0^x}(u)
		\quad \text{and} \quad
	\lambda_x = \sum_{i\in \Gamma_n} \varepsilon_i \alpha_{ix}.$$

Our main tool is the following lemma whose proof is exactly the same as the one of Lemma~4.1 in \cite{MargolisdelRioSerrano17}.

\begin{lemma}\label{LemmaLambdaAlpha}
	$u$ is rationally conjugate to $g_0$ if and only if $\lambda_i = \alpha_i$, for any positive integer $i$.
\end{lemma}

By Lemma~\ref{LemmaLambdaAlpha}, in order to achieve our goal it is enough to prove that $\lambda_i=\alpha_i$ for every positive integer $i$.
We argue by contradiction, so we assume that $\lambda_d\ne \alpha_d$ for some
positive integer $d$ which we assume to be minimal with this property.
Observe that if $\lambda_i = \alpha_i$ and $j$ is an integer such that $\gcd(i, n) =
\gcd(j, n)$, then there exists $\sigma \in \Gal(\Q(\alpha_1)/\Q)$ such that $\sigma(\alpha_i) = \alpha_j$ and applying $\sigma$ to the
equation $\lambda_i = \alpha_i$ we obtain $\lambda_j = \alpha_j$.  	
This implies that $d$ divides $n$. Note that $\alpha_1 = \lambda_1$, by our choice of $g_0$, and hence $d\ne 1$.
Moreover, $d\ne n$ because $\lambda_n = 2 \sum_{x\in \Gamma_n} \varepsilon_x = 2 = \alpha_n$ as the augmentation of $u$ is 1.

We claim that
\begin{equation}\label{KeyEquality}
\lambda_d = \alpha_d + d\tau \text{ for some } \tau \in \Z[\alpha_1].
\end{equation}
Indeed, for any $x\in\Gamma_{n}$ let $B_x=\varepsilon_x -1$ if $x\sim_n 1$ and $B_x = \varepsilon_x$ otherwise. Then for any integer $i$ we have $\lambda_i - \alpha_i = \sum_{x\in\Gamma_n} B_x \Tr_{\Q(\zeta) / \Q(\alpha_1)} \left(\zeta^{ix}\right)$.
The claim then follows by applying  Corollary~3.3
of \cite{MargolisdelRioSerrano17} with $F=\Q(\alpha_1)$, $R=\Z[\alpha_1]$ and $\omega_i = \lambda_i -\alpha_i$.
Observe that in the notation of that corollary $\Gamma_n=\Gamma_F$.

By \eqref{ChiExpresion} we have
\begin{equation}\label{Chidg0u}
\chi_d(g_0)=\sum_{\substack{i=0\\i\equiv d\bmod 2}}^d \alpha_i \quad \text{and} \quad
\chi_d(u) =  \sum_{x\in \Gamma_n} \varepsilon_x \chi_d(g_0^x) = \sum_{x\in \Gamma_n} \varepsilon_x \sum_{\substack{i=0\\i\equiv d \bmod 2}}^d\alpha_{ix} = \sum_{\substack{i=0\\i\equiv d\bmod 2}}^d \lambda_i.
\end{equation}
Combining this with \eqref{KeyEquality} and the minimality of $d$,  we deduce that $\chi_d(u)=\chi_d(g_0)+d\tau$.
Furthermore, $\tau\ne 0$, as $\lambda_d\ne \alpha_d$.
Therefore
\begin{equation}\label{dmu}
C_b(\chi_d(u)) \equiv  C_b(\chi_d(g_0)) \bmod d \quad \text{ for every } b\in \B
\end{equation}
and
\begin{equation}\label{Difference}
d\leq \left|C_{b_0}(\chi_d(u)) - C_{b_0}(\chi_d(g_0))\right| \quad \text{ for some }b_0\in \B.
\end{equation}

The bulk of our argument relies on an analysis of the eigenvalues of $\Theta_d(u)$ and the induction hypothesis on $n$ and $d$. More precisely, we will use \eqref{dmu} and \eqref{Difference} to obtain a contradiction by comparing the eigenvalues of $\Theta_d(g_0)$ and $\Theta_d(u)$.
Of course, we do not know the eigenvalues of the latter.
However we know the eigenvalues of $\Theta_d(u^c)$ for every $c\mid n$ with $c\ne 1$, because we know the eigenvalues of $\Theta_d(g_0)$ and $u^c$ is conjugate to $g_0^c$.
This provides information on the eigenvalues of $\Theta_d(u)$.
For example, recall that if $\xi$ is an eigenvalue of $\Theta_d(u)$ then $\xi$ and $\xi^{-1}$ have the same multiplicity as eigenvalues of $\Theta_d(u)$. Therefore, if $3\le h$ then the sum of the multiplicities of the eigenvalues of $\Theta_d(u)$ of order $h$ is even.
Moreover, for every $t$-regular element $g$ of $G$, the multiplicity of $1$ as eigenvalue of $\Theta_d(g)$ is congruent  modulo $2$ with the degree $d+1$ of $\chi_d$.
As $n$ is not a prime power there is an odd prime $p$ dividing $n$.
By the induction hypothesis $\Theta_d(u^p)$ is rationally conjugate to $\Theta_d(g_0^p)$.
Thus the multiplicity of $-1$ as eigenvalue of $\Theta_d(u^p)$ is even.
As the latter is the sum of the multiplicities as eigenvalues of $\Theta_d(u)$ of $-1$ and the elements of order $2p$, we deduce that the multiplicity of $-1$ as eigenvalue of $\Theta_d(u)$ is even.
Using this we can see that $\Theta_d(u)$ is conjugate to $\diag(\zeta^{\nu_{-d}},\zeta^{\nu_{2-d}},\dots, \zeta^{\nu_{d-2}}, \zeta^{\nu_d})$ for integers $\nu_{-d},\nu_{-d+2},\dots,\nu_{d-2},\nu_{d}$ such that $\nu_{-i}=-\nu_i$ for every $i$.
Let $X_d=\{i : 1\le i \le d, i\equiv d \bmod 2\}$.
Then, by \eqref{Chidg0u} and Proposition~\ref{BaseLemma}, we have for every $b\in \B$ that
\begin{equation}\label{DiferenciaCb}
\nonumber C_b(\chi_d(u))-C_b(\chi_d(g_0))=
\sum_{i\in X_d} \left(\kappa_{\nu_i}\cdot \beta_{b,\nu_i} \cdot \mu(\gamma(\nu_i))\cdot \delta_{b,\nu_i}^{(n/\bar{\gamma}(\nu_i))} -
\kappa_{i}\cdot \beta_{b,i} \cdot \mu(\gamma(i))\cdot \delta_{b,i}^{(n/\bar{\gamma}(i))}\right).
\end{equation}
Moreover, if $c\mid n$ with $c\ne 1$ then the lists $(c\nu_i)_{i\in X_d}$ and $(ci)_{i\in X_d}$ represent the same elements in $\Gamma_n$, up to ordering, and hence
$(\nu_i)_{i\in X_d}$ and $(i)_{i\in X_d}$ represent the same elements of $\Gamma_{\frac{n}{c}}$, up to ordering.
We express this by writing $(\nu(X_d))\sim_{\frac{n}{c}} (X_d)$.
This provides restrictions on $d$, $n$ and the $\nu_i$.

The following two lemmas are variants of Lemmas~4.2 and 4.3 of \cite{MargolisdelRioSerrano17}.

\begin{lemma}\label{KappaAndPrime}
\begin{enumerate}
\item\label{kappa1} Let $i\in X_d$. If $\kappa_i\ne 1$ then $n=2d$ and $i=d$.
If $\kappa_{\nu_i}\ne 1$ then $\frac{n}{d}$ is the smallest prime dividing $n$ and $\kappa_{\nu_j}=1$ for every $j\in X_d\setminus \{i\}$.
\item\label{PrimeGreaterThand} If $d> 2$ then $n$ is not divisible by any prime greater than $d$.
In particular, if $d$ is prime then $\kappa_{\nu_i}=1$ for every $i\in X_d$.
\end{enumerate}
\end{lemma}

\begin{proof}
Let $p$ denote the smallest prime dividing $n$.

\eqref{kappa1} The first statement is clear. Suppose that $\kappa_{\nu_i}\ne 1$. Then either $p=2$ and $\nu_i\equiv 0\bmod \frac{n}{2}$ or $\nu_i\equiv 0 \bmod n$.
As $(X_d)\sim_{\frac{n}{p}} (\nu(X_d))$ we deduce that $k\equiv 0 \bmod \frac{n}{p}$ for some $k\in X_d$. Therefore $d=k=\frac{n}{p}$ and for every $j\in X_d\setminus \{i\}$ we have $\nu_j\not\equiv 0 \bmod \frac{n}{p}$. Thus $\kappa_{\nu_j}=1$.

\eqref{PrimeGreaterThand} Suppose that $q$ is a prime divisor of $n$ with $d<q$.
Then $\frac{n}{d}\ne p$ and therefore, by \eqref{kappa1}, $\kappa_i=\kappa_{\nu_i}=1$ for every $i\in X_d$.
Thus, by \eqref{Difference} and \eqref{DiferenciaCb}, it is enough to show that $\delta^{(n/\bar{\gamma}(i))}_{b,i}\ne 0$ for at most one $i\in X_d$ and
$\delta^{(n/\bar{\gamma}(\nu_i))}_{b,\nu_i}\ne 0$ for at most one $i\in X_d$.
Observe that if $i\in X_d$ then $q\nmid i$ and hence $\frac{n}{\bar{\gamma}(i)}$ is multiple of $q$.
Moreover, if $i$ and $j$ are different elements of $X_d$ then $i$ and $j$ have the same parity and $-q<i-j<i+j<2q$. Therefore $i\not\sim_q j$.
Thus either $\delta^{\frac{n}{\bar{\gamma}(i)}}_{b,i}=0$ or $\delta^{\frac{n}{\bar{\gamma}(j)}}_{b,j}=0$.
As $(X_d)\sim_q (\nu(X_d))$, this also proves that $\delta^{\frac{n}{\bar{\gamma}(\nu_i)}}_{b,\nu_i}=0$ or $\delta^{\frac{n}{\bar{\gamma}(\nu_j)}}_{b,\nu_j}=0$.
\end{proof}

We obtain an upper bound for $\left|C_b(\chi_d(u)) - C_b(\chi_d(g_0))\right|$ in terms of the number of prime divisors $P(d)$ of $d$.

\begin{lemma}\label{BoundDifference}
For every $b\in\B$ we have
$\left|C_b(\chi_d(u)) - C_b(\chi_d(g_0))\right| \le 2 + 2^{P(d)+1}$.
\end{lemma}

\begin{proof}
Using \eqref{DiferenciaCb} it is enough to prove that $\sum_{i\in X_d} \kappa_i \delta_{b,i}^{(n/\bar{\gamma}(i))}\le 1 + 2^{P(d)}$ and $\sum_{i\in X_d} \kappa_{\nu_i} \delta_{b,\nu_i}^{(n/\bar{\gamma}(\nu_i))}\le 1 + 2^{P(d)}$.
This is a consequence of Lemma~\ref{KappaAndPrime}.\eqref{kappa1} and the following inequalities for every $e$ dividing $\PP{d}$:
$$\left| \left\{ i\in X_d : \gcd(d,\bar{\gamma}(i))=e, \delta_{b,i}^{(n / \bar{\gamma}(i))}=1 \right\} \right|  \leq  1, \quad
\left| \left\{ i\in X_d : \gcd(d,\bar{\gamma}(\nu_i))=e,
\delta_{b,\nu_i}^{(n / \bar{\gamma}(\nu_i))}=1 \right\} \right| \leq  1.
$$
We prove the second inequality, only  using that $(\nu(X_d))\sim_d (X_d)$.
This implies the first inequality by applying the second one to $u=g_0$.

Let $Y_e = \left\{ i\in X_d : \gcd(d,\bar{\gamma}(\nu_i))=e, \delta_{b,\nu_i}^{(n / \bar{\gamma}(\nu_i))}=1 \right\}$.
By changing the sign of some $\nu_i$'s, we may assume without loss of generality that if
$\delta_{b,\nu_i}^{(n/\bar{\gamma}(\nu_i))}=1$ then $b\equiv \nu_i \bmod \frac{n}{\bar{\gamma}(\nu_i)}$.
Thus, if $i\in Y_e$ then $b\equiv \nu_i \bmod \frac{n}{\bar{\gamma}(\nu_i)}$.
We claim that if $i,j\in Y_e$ then $\nu_i\equiv \nu_j \bmod d$.
Indeed, let $p$ be prime divisor of $d$.
If $n_p\ne d_p$ or $p\nmid e$ then clearly $\nu_i\equiv \nu_j \bmod d_p$.
Otherwise, i.e. $n_p=d_p$ and $p\mid e$, then $p$ divides both $\bar{\gamma}(\nu_i)$ and $\bar{\gamma}(\nu_j)$ and $\nu_i \equiv \nu_j \bmod \frac{d_p}{p}$.
Therefore, by Lemma~\ref{Elementary}.\eqref{LevelStep}, $\nu_i\equiv \nu_j \bmod n_p$, as desired.
As $(\nu(X_d))\sim_d (X_d)$ and the elements of $X_d$ represent different classes in $\Gamma_d$ we deduce that $|Y_e|\le 1$. This finishes the proof of the lemma.
\end{proof}

We are ready to finish the proof of Theorem~\ref{main}. Recall that we are arguing by contradiction.

By \eqref{Difference} and Lemma~\ref{BoundDifference} we have $d\le 2+2^{P(d)+1}$ and, using this, it is easy to show that $d\le 6$ or $d=10$. Indeed, if $P(d)\ge 3$ then $2+2^{P(d)+1}\ge d \ge 2\cdot 3 \cdot 5 \cdot 2^{P(d)-3}>14+2^{P(d)+1}$, a contradiction.
Thus $P(d)=2$ and $d\le 10$ or $P(d)=1$ and $d\le 5$.
Hence $d$ is either $2,3,4,5,6$ or $10$.
We deal with these cases separately.

\underline{Suppose that $d=2$.}
Then $\nu_2\sim_{n_p} 2$ for every prime $p$. By the assumptions on $n$ and Lemma~\ref{KappaAndPrime}.(\ref{kappa1}), this implies that $\kappa_2=\kappa_{\nu_2}=1$, $\gamma(2)=\gamma(\nu_2)$ and $\beta_{b_0,2}=\beta_{b_0,\nu_2}$.
Therefore $\left|C_{b_0}(\chi_2(u)) - C_{b_0}(\chi_2(g_0))\right|=
\left| \mu(\gamma(2))\left(\delta_{b_0,2}^{(n/\bar{\gamma}(2))}-\delta_{b_0,\nu_2}^{(n/\bar{\gamma}(\nu_2))}\right)
	\right|\le 1$, contradicting \eqref{Difference}.

\underline{Suppose that $d=3$}.
By Lemma~\ref{KappaAndPrime} and the assumptions on $n$, we have $\kappa_i=\kappa_{\nu_i}=1$ for every $i\in X_3$ and $\PP{n}= 6$.
If $2^4\mid n$ or $3^2\mid n$ then $\left| \left\{ i=1,3 : \delta_{b,i}^{(n/\bar{\gamma}(i))}=1 \right\} \right| \leq 1$ and $\left| \left\{ i=1,3 : \delta_{b,\nu_i}^{(n/\bar{\gamma}(\nu_i))}=1 \right\} \right| \leq 1$, which implies $\left|C_{b_0}(\chi_3(u)) - C_{b_0}(\chi_3(g_0))\right|\le 2$, contradicting (\ref{Difference}).
Thus $n=24$, since $n$ is neither $12$ nor  a prime power and it is multiple of $4$.
In this case we have $\bar{\gamma}(1)=\gamma(1)=2$, $\bar{\gamma}(3)=\gamma(3)=3$, $\beta_{b,1}=\beta_{b,3}=1$ and  $C_b(\chi_3(g_0))=-\delta_{b,1}^{(12)} -\delta_{b,3}^{(8)}$ for every $b\in\B$.
We may assume that $3\mid \nu_3$ and $3\nmid \nu_1$ because $(\nu(X_3))\sim_3 (X_3)$.
Suppose that $\nu_3 \sim_8 3$ and $\nu_1 \sim_8 1$. Then $\bar{\gamma}(\nu_{1})=\gamma(\nu_1)=2$, $\bar{\gamma}(\nu_3)=\gamma(\nu_3)=3$, $\beta_{b_0,\nu_3}=1$ and $\delta_{b_0,3}^{(8)} = \delta_{b_0,\nu_3}^{(8)}$, which implies $\left|C_{b_0}(\chi_3(u)) - C_{b_0}(\chi_3(g_0))\right| \leq 2$, contradicting (\ref{Difference}).
Suppose now that  $\nu_3 \sim_8 1$ and $\nu_1 \sim_8 3$. This implies that $\nu_1 \equiv \pm 3 \bmod 8$ and $\nu_1 \equiv \pm 1 \bmod 3$ (because $3\mid \nu_3$ but $3\nmid \nu_1$). Thus either $\nu_1 \equiv \pm 11\bmod 24$ or $\nu_1\equiv \pm 5 \bmod 24$.
As $(\nu(X_3))\sim_{12}(X_3)$, we deduce that the only possibility is $\nu_1 \equiv \pm 11 \bmod 24$.
In this case we have $\bar{\gamma}(\nu_1)=\gamma(\nu_1)=1$ and $\bar{\gamma}(\nu_3)=\gamma(\nu_3)=6$.
Hence $C_{11}(\chi_3(u))-C_{11}(\chi_3(g_0))=\delta_{11,\nu_{1}}^{(24)} + \delta_{11,\nu_3}^{(4)} + \delta_{11,1}^{(12)} + \delta_{11,3}^{(8)} = 4$,   contradicting (\ref{dmu}).

\underline{Suppose that $d=4$.}
By Lemma~\ref{KappaAndPrime} and the assumptions on $n$, we have $\kappa_i=\kappa_{\nu_i}=1$ for every $i\in X_4$ and $\PP{n}=6$.
If $3^3\mid n$ or $2^3\mid n$ then $\left| \left\{ i=2,4 : \delta_{b,i}^{(n/\bar{\gamma}(i))}=1 \right\} \right| \leq 1$ which implies
$\left|C_{b_0}(\chi_4(u)) - C_{b_0}(\chi_4(g_0))\right|\le 3$, contradicting (\ref{Difference}).
Thus  $n=36$.
In this case we have $\gamma(2)=1=\beta_{b_0,2}=\beta_{b_0,4}$ and $\gamma(4)=2$, which
implies $|C_{b_0}(\chi_4(g_0))|\le 1$ and hence $\left|C_{b_0}(\chi_4(u)) - C_{b_0}(\chi_4(g_0))\right|\le 3$, contradicting again  (\ref{Difference}).

\underline{Suppose that $d=5$}.
Since $(\nu(X_5)) \sim_5 (X_5)$, there is exactly one $\nu_i$ which is divisible by $5$, say $\nu_5$.
In particular, for $i\ne 5$ we have $5\mid \frac{n}{\bar{\gamma}(\nu_i)}$ and $5\mid \frac{n}{\bar{\gamma}(i)}$.
Moreover, if $j$ is an integer not multiple of $5$ then $|\{i=1,3 : \nu_i \sim_5 j\}|\leq 1$. This implies that
$\left| \left\{ i=1,3 : \delta_{b,i}^{(n/\bar{\gamma}(i))}=1 \right\} \right| \leq 1$ and $\left| \left\{ i=1,3 : \delta_{b,\nu_i}^{(n/\bar{\gamma}(\nu_i))}=1 \right\} \right| \leq 1$.
On the other hand, as $n\ne 10$, we deduce that $\kappa_i=1$ for every $i\in X_5$, by Lemma~\ref{KappaAndPrime}.(\ref{kappa1}).
Therefore, using (\ref{Difference}) and (\ref{DiferenciaCb}), we deduce that  $\kappa_{\nu_5}=2$, in contradiction with Lemma~\ref{KappaAndPrime}.\eqref{kappa1}.
	
\underline{Suppose that $d=6$}.
By Lemma~\ref{KappaAndPrime}, we have $\PP{n}\mid 30$ and $\kappa_i=\kappa_{\nu_i}=1$ for every $i\in X_6$ because $n\ne 12$.
If $25\mid n$, or $9\mid n$ or $8\mid n$ then we have
$\left|\left\{i=2,4,6 : \delta_{b,i}^{(n/\bar{\gamma}(i))}=1\right\}\right|\leq 2$ and
$\left|\left\{i=2,4,6 : \delta_{b,\nu_i}^{(n/\bar{\gamma}(\nu_i))}=1\right\}\right|\leq 2$.
This implies that $\left|C_{b_0}(\chi_6(u)) - C_{b_0}(\chi_6(g_0))\right| \leq 4$, yielding a contradiction with (\ref{Difference}).
Therefore  $n=60$ and hence $\beta_{b,2}=\beta_{b,4}=\beta_{b,6}=1$, $\bar{\gamma}(2)=1$, $\bar{\gamma}(4)=\gamma(4)=2$ and $\bar{\gamma}(6)=\gamma(6)=3$.
This implies that $|C_{b_0}(\chi_6(g_0))|\le 2$ and hence $\left|C_{b_0}(\chi_6(u)) - C_{b_0}(\chi_6(g_0))\right|\leq 5$, yielding a contradiction with (\ref{Difference}).
	
\underline{Suppose that $d=10$}.
If $5\nmid \frac{n}{\bar{\gamma}(i)}$ for some $i\in X_{10}$ then $n_5=(\bar{\gamma}(i))_5=5$ and hence $5\mid i$. The same also holds for $\nu_i$.
Therefore, if $5\nmid i$ then $5\mid \frac{n}{\bar{\gamma}(i)}$ and if $5\nmid \nu_i$ then $5\mid \frac{n}{\bar{\gamma}(\nu_i)}$.
Thus $\left|\left\{i\in X_{10}: 5\nmid i, \delta_{b,i}^{(n/\bar{\gamma}(i))}=1\right\}\right|\leq 2$ and
$\left|\left\{i\in X_{10} : 5\nmid \nu_i,\delta_{b,\nu_i}^{(n/\bar{\gamma}(\nu_i))}=1\right\}\right|\leq 2$. 	 This implies that $\left|C_{b_0}(\chi_{10}(u)) - C_{b_0}(\chi_{10}(g_0))\right|\leq 8$, contradicting (\ref{Difference}).
	
\bibliographystyle{alpha}
\bibliography{SL}
\end{document}